\documentclass[a4paper]{article}

\usepackage[utf8]{inputenc}

\usepackage{lmodern}
\usepackage{array}
\usepackage{graphicx}
\usepackage{subfig}
\usepackage{amssymb, amsfonts, amsthm, amsmath}
\usepackage{enumerate}

\usepackage[english]{babel}

\usepackage[colorlinks]{hyperref}

\usepackage{tikz}
\usetikzlibrary{decorations}

\usetikzlibrary{decorations.pathreplacing}

\tikzset{individu/.style={draw,thick}}

\numberwithin{equation}{section}

\theoremstyle{plain}
\newtheorem{theorem}{Theorem}[section]
\newtheorem{corollary}[theorem]{Corollary}

\newtheorem{lemma}[theorem]{Lemma}
\newtheorem{proposition}[theorem]{Proposition}

\theoremstyle{definition}

\theoremstyle{remark}
\newtheorem{remark}[theorem]{Remark}

\newcommand{\N}{\mathbb{N}}
\newcommand{\Z}{\mathbb{Z}}

\newcommand{\frakC}{\mathfrak{C}}

\newcommand{\ind}[1]{\mathbf{1}_{\left\{#1\right\}}}

\DeclareMathOperator{\E}{\mathbf{E}}

\renewcommand{\P}{\mathbf{P}}

\renewcommand{\epsilon}{\varepsilon}
\renewcommand{\phi}{\varphi}

\newcommand{\Addresses}{{
  \bigskip
  \footnotesize

  Bastien Mallein, \textsc{LAGA - Institut Galil\'ee, 99 avenue Jean-Baptiste Cl\'ement
93430 Villetaneuse, France}\par\nopagebreak
  \textit{E-mail address}: \texttt{mallein@math.univ-paris13.fr}
  
  \medskip

  Sanjay Ramassamy, \textsc{Unit\'e de Math\'ematiques Pures et Appliqu\'ees, \'Ecole normale sup\'erieure de Lyon, 46 all\'ee d'Italie, 69364 Lyon Cedex 07, France}\par\nopagebreak
  \textit{E-mail address}: \texttt{sanjay.ramassamy@ens-lyon.fr}

}}

\title{Two-sided infinite-bin models and analyticity for Barak-Erd\H{o}s graphs}
\author{Bastien Mallein \and Sanjay Ramassamy}
\date{\today}

\newcommand{\calP}{\mathcal{P}}

\begin{document}

\maketitle

\begin{abstract}
In this article, we prove that for any probability distribution $\mu$ on $\mathbb{N}$ one can construct a two-sided stationary version of the infinite-bin model --an interacting particle system introduced by Foss and Konstantopoulos-- with move distribution $\mu$. Using this result, we obtain a new formula for the speed of the front of infinite-bin models, as a series of positive terms. This implies that the growth rate $C(p)$ of the longest path in a Barak-Erd\H{o}s graph of parameter $p$ is analytic on $(0,1]$. 
\end{abstract}

\section{Introduction and main results}

This article introduces a new approach to the study of infinite-bin models, a family of interacting particle systems introduced by Foss and Konstantopoulos \cite{FK} which yields new results not only for this particle system but also for Barak-Erd\H{o}s graphs \cite{BE}, which are a natural class of random directed acyclic graphs.

Roughly speaking, the infinite-bin model is a random discrete-time dynamics on configurations of balls in an infinite row of bins, where a new ball is added inside some bin at each step of time according to some random rule, where the randomness is governed by a probability distribution $\mu$ on $\N$. Each configuration has a well-defined notion of front (a non-empty bin such that all the bins to its right are empty) and one of the most interesting observables in this model is the speed at which the front moves to the right. In this article we construct for every probability distribution $\mu$ on $\N$ a two-sided stationary version of the infinite-bin model (time is indexed by $\Z$ rather than $\Z_+$) and we use this construction to express the speed of the front as a series of positive terms.

The Barak-Erd\H{o}s graph with edge probability $p$ is a directed acyclic version of the classical Erd\H{o}s-R\'enyi random graph with edge probability $p$ \cite{ER}. Foss and Konstantopoulos \cite{FK} introduced a coupling between the Barak-Erd\H{o}s graph with edge probability $p$ and the infinite-bin model where $\mu$ is the geometric distribution of parameter $p$, whereby the growth rate $C(p)$ of the length of the longest directed path in the Barak-Erd\H{o}s graph with edge probability $p$ equals the speed of the front of the infinite-bin model associated with the geometric distribution of parameter $p$. Using the above-mentioned series formula for the speed of infinite-bin models, we manage to prove that the function $C(p)$ is analytic for $0< p \leq 1$.

In the rest of the introduction, we first describe Barak-Erd\H{o}s graphs and we state the analyticity result for $C(p)$. Then we formally introduce the infinite-bin model, we state the results about the existence of a two-sided stationary version and a formula for the speed of the front and finally we discuss the connection with other probabilistic models.

\subsection{Barak-\texorpdfstring{Erd\H{o}s}{Erdos} graphs}

Given an integer $n\geq1$ and a parameter $0 \leq  p \leq 1$, the Barak-Erd\H{o}s graph $G_{n,p}$ is the graph with vertex set $\left\{1,\ldots,n\right\}$ obtained by adding an edge directed from $i$ to $j$ with probability $p$ for every pair $(i,j)$ with $1 \leq i <j  \leq n$, independently for each pair. This model was introduced by Barak and Erd\H{o}s~\cite{BE} and has since then been widely considered. The most studied feature of these graphs $G_{n,p}$ has been the length of their longest path $L_n(p)$, with applications including food chains~\cite{CN,NC}, the speed of parallel processes~\cite{GNPT,IN}, last passage percolation~\cite{FMS} and the stability of queues~\cite{FK}. Some extensions of the model were considered in~\cite{DFK,KT}.

Newman~\cite{N} proved that there exists a function $C: [0,1] \to [0,1]$ such that for any $0 \leq p \leq 1$,
\begin{equation}
  \label{eqn:defC}
  \lim_{n \to \infty} \frac{L_n(p)}{n} = C(p) \quad \text{in probability.}
\end{equation}
Moreover he showed that the function $C$ is continuous, differentiable at $0$ and that $C'(0) = e$ (see Figure~\ref{fig:cpgraph} for a plot of $C(p)$).

\begin{figure}[htbp]
\centering
\includegraphics[height=2in]{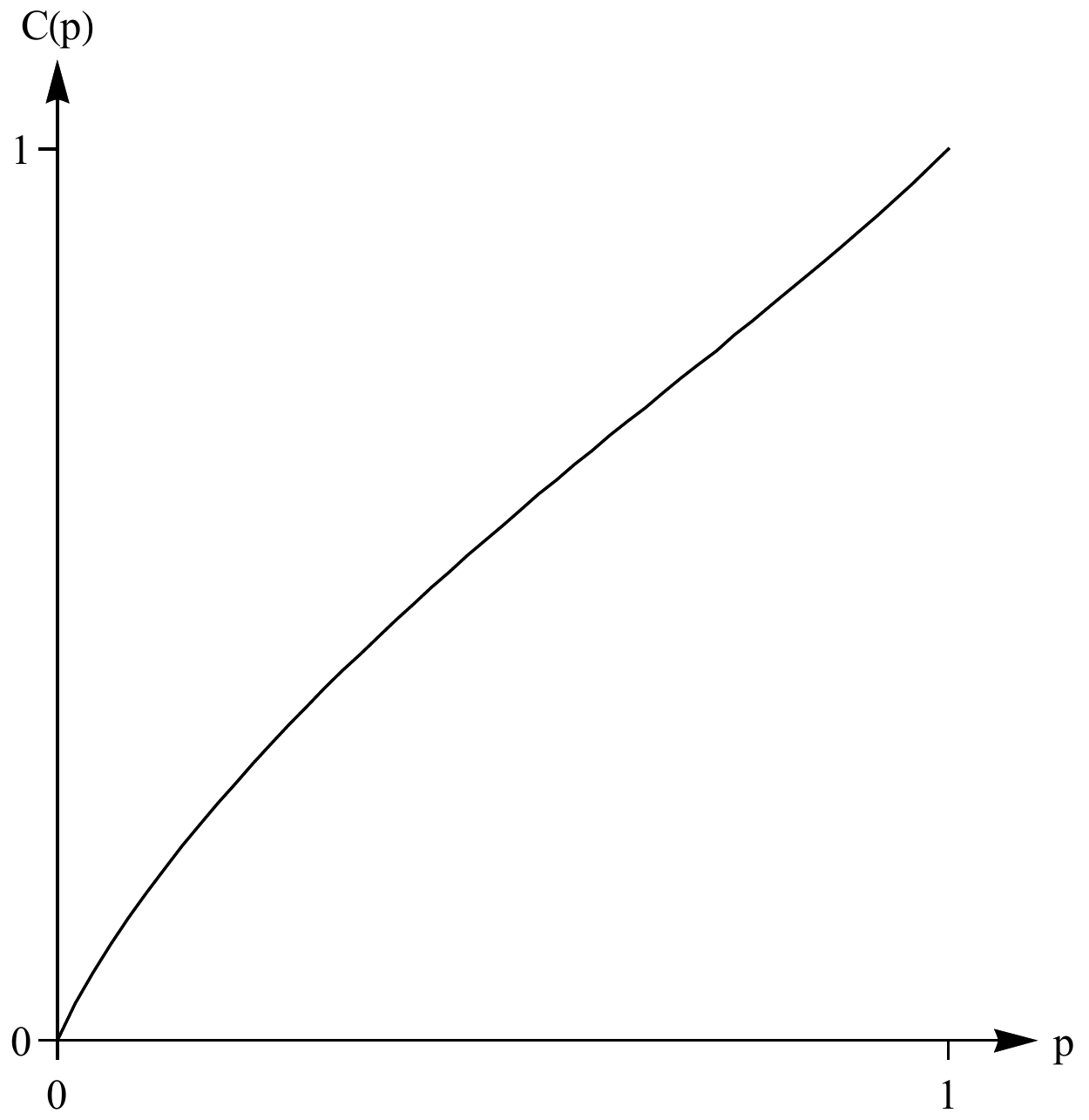}
\caption{Plot of a simulation of $C(p)$.}
\label{fig:cpgraph}
\end{figure}

Using the coupling with infinite-bin models mentioned above, Foss and Konstantopoulos~\cite{FK} obtained upper and lower bounds for the function $C$ which are tight in a neighborhood of $1$. In~\cite{MR}, we proved that $C$ is an analytic function on $(\frac{1}{2},1]$ and showed that the power series expansion of $C(p)$ centered at $1$ has integer coefficients. Moreover, we proved that
\[
  C(p)  = pe \left( 1 - \frac{\pi^2(1 + o(1))}{2 (\log p)^2}\right) \text{ as } p \to 0.
\]
In particular, this implied that $C$ has no second derivative at $p=0$. We raised the question whether there exists a phase transition for some $0<p_0<1$, where the function $C$ stops being analytic. In this paper we provide a negative answer to this question.

\begin{theorem}
\label{thm:analyticity}
The function $p \mapsto C(p)$ is analytic on $(0,1]$.
\end{theorem}

The proof of this theorem is based on a series formula for $C(p)$ which follows from a coupling between Barak-Erd\H{o}s graphs and a specific subclass of infinite-bin models. The next subsection is devoted to general infinite-bin models.

\subsection{The infinite-bin model}

The infinite-bin model is an interacting particle system on $\Z$ that can be described as follows. Consider a set of bins indexed by $\Z$, each bin containing a finite number of balls. There is a well-defined notion of front, which is a non-empty bin such that all the bins to its right are empty and the bins to its left are non-empty. Given an i.i.d. sequence $(\xi_n, n \geq 1)$ of positive integers, the infinite-bin model evolves as follows: at every time $n$, a new ball is added to the bin immediately to the right of the bin containing the $\xi_n$th rightmost ball. We denote by $\mu$ the distribution on $\N$ of the random variable $\xi_1$.

This process was introduced by Foss and Konstantopoulos~\cite{FK} in order to study Barak-Erd\H{o}s graphs and further studied in~\cite{FZ,CR,MR}. Note that in the original description~\cite{FK}, the sequence $(\xi_n)$ is only supposed to be stationary and ergodic, not necessarily i.i.d. Constructing the stationary version of an infinite-bin model could be done in ergodic settings as well under mild assumptions, but we choose to stick to the i.i.d. setting to keep the proofs simple.

We introduce some notation to define infinite-bin models more precisely. A \emph{configuration} of balls $X$ is defined to be any collection $(X(k), k \in \Z) \in \Z_+^\Z$ (where $X(k)$ represents the number of balls in the bin of index $k$) such that there exists $F(X) \in \Z$ satisfying the following two conditions:
\begin{enumerate}
  \item every bin with an index smaller or equal to $F(X)$ is non-empty ;
  \item every bin with an index strictly larger than $F(X)$ is empty.
\end{enumerate}
The index $F(X)$ of the rightmost non-empty bin is called the position of the \emph{front} of the configuration $X$. We denote by $S$ the set of configurations. Note that one could allow bins containing infinitely many balls or even empty bins to the left of the front, provided the total number of balls in a configuration is infinite, and the process would remain well-defined. However, the space $S$ is stable under the dynamics of the infinite-bin model, and we will only consider infinite-bin models as Markov processes on $S$, again for the sake of simplifying the proofs.

Let $k \geq 1$ be an integer, we define a \emph{move of type} $k$ as a map $\Phi_k$ from the set of configurations to itself. For any $X \in S$ and $k \in \N$, we set
\[
N(X,k) = \sum_{j \geq k} X(j) \quad \text{and} \quad B(X,k) = \sup\{ n \in \N: N(X,n) \geq k\},
\]
which are respectively the number of balls in or to the right of bin number $k$ and the index of the bin containing the $k$th rightmost ball. We define
\[
  \Phi_k: X \in S \mapsto (X(j) + \ind{j=B(X,k)+1}, j \in \Z)\in S.
\]
In other words, $\Phi_k(X)$ is the configuration obtained from $X$ by adding one ball to the bin immediately to the right of the bin containing the $k$th rightmost ball.

Let $\mu$ be a probability distribution on $\N$ and $X_0$ be an initial configuration, which may be deterministic or random. We construct $(X_n)_{n \geq 0}$ the infinite-bin model with move distribution $\mu$ (or IBM($\mu$) for short) as the following stochastic recursive sequence in $S$:
\[
  \forall n \geq0, \quad X_{n+1} = \Phi_{\xi_{n+1}}(X_n),
\]
where $(\xi_n)_{n\geq1}$ is an i.i.d. sequence of random variables of law $\mu$. Foss and Konstantopoulos~\cite{FK} proved that when $\mu$ has finite expectation there exists a constant $v_\mu \in [0,1]$ independent of the initial configuration $X_0$ such that
\begin{equation}
  \label{eqn:defineSpeed}
  \lim_{n \to \infty} \frac{F(X_n)}{n} = v_\mu \quad \text{a.s. and in } L^1.
\end{equation}
In~\cite{MR}, we proved that this result holds without any assumption on the measure~$\mu$. The constant $v_\mu$ is called the speed of the IBM($\mu$).

In this article, we express the speed $v_\mu$ as the sum of a series with positive terms for a general move distribution $\mu$. This series representation of $v_\mu$ is based on the appearance of special patterns in the sequence $(\xi_n, n \geq 1)$. To write it down more precisely, we introduce some notation on finite patterns.

We denote by
\[ \mathcal{W} = \bigcup_{n \geq 0} \N^n \]
the set of all finite sequences of positive integers. Sequences in $\mathcal{W}$ will simply be called \emph{words} and an element of a word will be called a \emph{letter}. By convention $\emptyset$ is the only element of $\N^0$, called the \emph{empty word}. Given a word $\alpha \in \mathcal{W}$, we denote by $|\alpha|$ the length of $\alpha$ (i.e. its number of letters), and for $1 \leq k \leq |\alpha|$, by $\alpha_k$ its $k$th letter. Furthermore, if $I$ is an interval (possibly infinite) of $\Z$  and $\alpha=(\alpha_i)_{i \in I}$ is a sequence of positive integers indexed by $I$, for any $n_1\leq n_2 \in I$ we denote by $\alpha_{n_1}^{n_2}$ the word $(\alpha_{n_1},\alpha_{n_1+1},\ldots,\alpha_{n_2})$ of length $n_2-n_1+1$. Finally, for two words $\alpha$ and $\beta$, the word $\alpha\cdot \beta$ is defined to be the concatenation of $\alpha$ and $\beta$, i.e. if $\alpha = (\alpha_1,\ldots, \alpha_n) \in \N^n$ and $\beta =(\beta_1,\ldots,\beta_p) \in \N^p$, we set
\[
  \alpha \cdot \beta = (\alpha_1,\ldots, \alpha_n,\beta_1,\ldots, \beta_1) \in \N^{n+p}.
\]

Recall that the map $\Phi_k$ denotes a single move of type $k$. We extend the notation by defining the map $\Phi_\alpha$ for every $\alpha = (\alpha_1,\ldots,\alpha_n) \in \mathcal{W}$ by
\begin{equation}
  \label{eqn:defPhiword}
  \forall X \in S, \quad  \Phi_\alpha(X) = \left(\Phi_{\alpha_n} \circ \Phi_{\alpha_{n-1}} \circ \cdots \circ \Phi_{\alpha_1} \right)(X).
\end{equation}
In other words, $\Phi_\alpha(X)$ is the configuration obtained from $X$ by successively applying the moves of type $\alpha_1,\alpha_2,\ldots,\alpha_n$. Using this notation, we define for every $X \in S$ the set of $X$\emph{-good words} as
\[
  \mathcal{P}_X= \left\{ \alpha \in \mathcal{W} \backslash \{ \emptyset \}: F(\Phi_\alpha(X)) > F(\Phi_{\alpha_1^{|\alpha|-1}}(X)) \right\},
\]
i.e. the set of finite sequences of moves such that, starting from $X$, the final move makes the front advance to the right by one unit, by adding a ball in a previously empty bin.

We define the set of \emph{good words} to be the words that are $X$-good for every starting configuration $X$, as well as the set of \emph{bad words} to be the words that are $X$-good for no initial configuration $X$, i.e.
\begin{equation}
  \label{eqn:defGoodandBad}
  \mathcal{G}= \bigcap_{X \in S} \calP_X \quad \text{and} \quad  \mathcal{B}= \bigcap_{X \in S} \calP_X^c.
\end{equation}
Observe that with these definitions, while the complement of $X$-good words is $X$-bad words, the complement of $\mathcal{G}$ is larger than $\mathcal{B}$, i.e. $\mathcal{G} \cup \mathcal{B} \subsetneq \mathcal{W}$. In other terms, there are three types of words: good words, bad words and words that are neither good nor bad. For example, the words $(1)$ and $(1,1)$ are good, the words $(1,2)$ and $(2,1,2)$ are bad and the word $(2,2)$ is neither good nor bad. Finally, we define the set of \emph{minimal good words} (resp. \emph{minimal bad words}) as the good (resp. bad) words that have no good (resp. bad) strict suffix:
\begin{align}
  \label{eqn:defMinGoodandBad}
  \mathcal{G}_m &= \left\{ \alpha \in \mathcal{G}: \forall\ 2 \leq k \leq |\alpha|, \alpha_k^{|\alpha|} \not \in \mathcal{G} \right\}  \\
 \mathcal{B}_m &= \left\{ \alpha \in \mathcal{B}: \forall\ 2 \leq k \leq |\alpha|, \alpha_k^{|\alpha|} \not \in \mathcal{B} \right\}.
\end{align}
For example, the word $(1)$ (resp. $(1,2)$) is minimal good (resp. minimal bad) while the word $(1,1)$ (resp. $(2,1,2)$) is not minimal good (resp. not minimal bad).

A probability distribution $\mu$ on $\N$ is called \emph{non-degenerate} if it is not a Dirac mass, i.e. if its support contains at least two elements. We obtain in this article the following formula, which holds for the speed of the front of any infinite-bin model whose move distribution is non-degenerate.

\begin{theorem}
\label{thm:main}
Fix a non-degenerate probability distribution $\mu$ on $\N$. For any $\alpha\in\mathcal{W}$, we set $w_\mu(\alpha) = \prod_{j=1}^{|\alpha|} \mu(\alpha_j)$. Then we have
\begin{equation}
\label{eq:speedformula}
  v_\mu = \sum_{\alpha \in \mathcal{G}_m} w_\mu(\alpha) = 1 - \sum_{\alpha \in \mathcal{B}_m} w_\mu(\alpha).
\end{equation}
\end{theorem}

Note that for $\mu=\delta_k$, the Dirac mass at $k$, the infinite-bin model is deterministic and $v_{\delta_k}=1/k$, but for all $k \geq 2$ the equalities~\eqref{eq:speedformula} do not hold.

For non-degenerate $\mu$, we also remark that from \eqref{eq:speedformula}, we deduce that
\[
  \sum_{\alpha \in \mathcal{G}_m} w_\mu(\alpha) + \sum_{\alpha \in \mathcal{B}_m} w_\mu(\alpha) = 1,
\]
which can be interpreted as follows: given $(\xi_{-n}, n \geq 0)$ a family of i.i.d. non-degenerate integer-valued random variables, almost surely there exists $n \geq 0$ such that $\xi_{-n}^0$ is either a good or a bad word. This is indeed a key step of the proof and is a straightforward consequence of Proposition \ref{prop:coupling}. In the remainder of the article, every probability distribution on $\N$ will be assumed to be non-degenerate, unless otherwise stated.

Formula~\eqref{eq:speedformula} for the speed of the infinite-bin model can be compared to the one we obtained in~\cite{MR}. For every $X\in S$, we defined the map
\[
\epsilon_X: \alpha \in \mathcal{W}\setminus \{\emptyset\} \mapsto \ind{\alpha \in \mathcal{P}_X}-\ind{\alpha_2^{|\alpha|} \in \mathcal{P}_X} \in \{-1,0,1\}.
\]
We showed in~\cite{MR} that whenever the series
\begin{equation}
\label{eq:oldformula}
  \sum_{\alpha \in \mathcal{W}} \epsilon_X(\alpha) w_\mu(\alpha) 
\end{equation}
converges absolutely, then its sum is equal to $v_\mu$. However there was no clear condition on~$\mu$ for the series~\eqref{eq:oldformula} to converge. We only managed to prove its convergence for probability distributions with light enough tails, such as geometric distributions with parameter $p>1/2$. By contrast, the new formula~\eqref{eq:speedformula} is more tractable, as it only has positive terms and it holds for every non-degenerate probability distribution $\mu$. However, formula~\eqref{eq:oldformula} is still well-adapted for explicit estimates, as the computation of $\epsilon(\alpha)$ is linear in $|\alpha|$, while verifying that a word $\alpha$ belongs to $\mathcal{G}$ has a complexity which is exponential in the largest letter of $\alpha$.

Theorem~\ref{thm:main} is based on the construction of a two-sided stationary version of the infinite-bin model, i.e. a process for which time takes values in $\Z$ rather than in $\Z_+$. More precisely, we define
\[
  \Psi_r: X \in S \mapsto (X(r+j), j \in \Z) \in S,
\]
the shift operator on $S$, which shifts all the balls by $r$ units to the left. Then the following result holds.
\begin{theorem}
\label{thm:biinfinite}
Let $(\xi_n , n \in \Z)$ be a family of i.i.d. random variables with a non-degenerate distribution $\mu$, we set $\mathcal{F}_n=\sigma(\xi_k, k \leq n)$. Almost surely, there exists a unique process $(Y_n, n \in \Z)$ on $S$ such that the following three conditions hold:
\begin{itemize}
 \item $F(Y_0)=0$ ;
 \item $\forall n \in \Z, \ Y_{n+1} = \Phi_{\xi_{n+1}}(Y_n)$ ;
 \item $\Psi_{F(Y_n)}(Y_n) \in \mathcal{F}_n$.
\end{itemize}
\end{theorem}

Note that $\Psi_{F(Y)}(Y)$ is the configuration $Y$ translated such that the front is at position $0$. We call the process $(Y_n)$ a stationary version of the infinite-bin model as the process $(\Psi_{F(Y_n)}(Y_n), n \in \Z)$ is a stationary Markov process. In other words, $(Y_n)$ depicts a wave of balls moving from left to right, such that the law of the wave considered up to translation is stationary.

In \cite{FK}, Foss and Konstantopoulos proved the existence of a two-sided stationary version of the infinite-bin model in the case when $\mu$ has finite expectation (the general framework of extended renovation theory which they developed actually also encompasses some cases of light-tailed $\mu$ with infinite expectation). They showed in that case that if one samples an infinite-bin model $(X_n)_{n\geq0}$ and a two-sided process $(Y_n)_{n\in\Z}$ using the same sequence $(\xi_n)_{n\in\Z}$, then $(X_n)$ coupling-converges to $(Y_n)$ which entails the joint convergence of the number of balls in bins within a fixed finite distance from the front. Their construction was based on going back in time and searching for certain renovation events, which determine where all the balls are placed after the renovation event starts, regardless of what the configuration was before the start of the event. These renovation events have positive probability when $\mu$ has finite expectation, but have probability zero otherwise. The renovation events they considered correspond to suffixes of $(\xi_n)_{n\in\Z}$ which are infinite on the right and finite on the left, such that the $i$-th letter is at most equal to $i$.

In order to construct $(Y_n)_{n\in\Z}$ even when $\mu$ has infinite expectation, we consider another class of words, described in Section~\ref{sec:coupling}, which have positive probability even when $\mu$ has infinite expectation. Namely we observe that there exist almost surely finite suffixes of $(\xi_n)_{n\in\Z_{\leq0}}$ which determine the content of a finite number of bins at the front at time $0$, regardless of what the configuration was before the appearance of that suffix. This observation makes it possible to do perfect simulation from the stationary measure of any infinite-bin model, in the spirit of what has been done for other processes with long memory \cite{CFF,FK,PW}.

One of the main reasons for the study of infinite-bin models is the connection with Barak-Erd\H{o}s graphs, which holds only when $\mu$ is a geometric distribution. For $p \in [0,1]$, we denote by $\mu_p$ the geometric distribution of parameter $p$. Foss and Konstantopoulos \cite{FK} introduced a coupling between the infinite-bin model with moves distributed like $\mu_p$ and Barak-Erd\H{o}s graph of parameter $p$ by observing that, as one grows a Barak-Erd\H{o}s graph by adding vertices one by one, recording the length of the longest path ending at each vertex produces a process distributed like the IBM($\mu_p$), see also \cite[Section 5]{MR} for more details. The Foss-Konstantopoulos coupling implies in particular that
\begin{equation}
\label{eq:FKcoupling}
\forall p \in [0,1], \quad v_{\mu_p}=C(p).
\end{equation}
As a consequence of Theorem~\ref{thm:main}, we immediately deduce the following formula for the growth rate $C(p)$ of the length of the longest path in Barak-Erd\H{o}s graphs with edge probability $p$:
\begin{corollary}
\label{cor:Cpformula}
For every $0 \leq p \leq 1$,
\begin{equation}
C(p)=\sum_{\alpha \in \mathcal{G}_m} p^{|\alpha|} (1-p)^{\sum_{j=1}^{|\alpha|} (\alpha_j - 1)}.
\end{equation}
\end{corollary}

Another special case of the infinite-bin model can be coupled with a known stochastic process. As observed in~\cite{MR}, the speed $w_k$ of an infinite-bin model with measure $\nu_k$ uniform on $\{1,\ldots, k\}$ is the same (up to a factor $k$) as the speed of a continuous-time branching random walk on $\Z$ with selection of the rightmost $k$ individuals, with a specific reproduction law. In the case of a general branching random walk, each individual reproduces after an exponential time, whereby it gives birth to a random number of children placed at random locations around the parent and the parent dies immediately after. The randomness is governed by the \emph{reproduction law}. The reproduction law corresponding to the infinite-bin model is the one where each parent has exactly two children, one placed one unit to its right and one placed at the same position as the parent (which serves to replace the dead parent). The particular infinite-bin model with uniform distribution was first studied by Aldous and Pitman in~\cite{AP}, who proved that $\lim_{k \to \infty} kw_k = e$. Denoting by $\mathcal{G}_{m,k}$ the set of minimal good words using letters only between $1$ and $k$, formula~\eqref{eq:speedformula} yields
\begin{equation}
\label{eq:branchingrwspeed}
  w_k = \sum_{\alpha \in \mathcal{G}_{m,k}} \frac{1}{k^{|\alpha|}}.
\end{equation}

\begin{remark}[Speed of a branching random walk with selection]
\label{rem:speed}
The asymptotic behavior of the speed of a branching random walk with selection with a general reproduction law was conjectured in Brunet and Derrida~\cite{BD}. In the special case of the infinite-bin model with uniform distribution, this conjecture can be stated as
\begin{equation}
  \label{eqn:conjectureBD}
  kw_k = e -\frac{e\pi^2}{2(\log k+3\log \log k + o(\log \log k))^2} \quad \text{ as } k \to \infty.
\end{equation}
So far, two terms of the asymptotic behaviour of the speed of branching random walks with selection have been obtained by B\'erard and Gou\'er\'e \cite{BeG10} for binary reproduction laws and extended in \cite{Mal15a} to more general reproduction laws. In the special case of the infinite-bin model with uniform distribution, these results imply that
\[
 k w_k = e - \frac{e \pi^2}{2(\log k)^2}(1+ o(1)) \quad \text{as } k \to \infty,
\]
see \cite[Lemma 7.1]{MR}. It would be interesting to prove the Brunet-Derrida conjecture with the additional term in $\log \log k$ in the special case of the infinite-bin model with uniform distribution using formula \eqref{eq:branchingrwspeed}, or to use the conjectured formula \eqref{eqn:conjectureBD} to gain information on the distribution of good words, by performing an analysis of singularities.
\end{remark}

One may extend the above connection between branching random walks and infinite-bin models to infinite-bin models with general move distribution $\mu$, by seeing them as some rank-biased branching random walks, in which the $k$th rightmost particle reproduces at each time step with probability $\mu(k)$, by giving birth to a new child at distance $1$ to its right. This is perhaps more striking when considering the infinite-bin model in continuous time, such that each new ball appears after an exponential random time of parameter $1$. Then the branching random walk can be described as follows: each particle reproduces independently by making a new child at distance $1$ to its right at rate $\mu(k)$ if the particle is the $k$th rightmost particle. Therefore, the rate at which particles reproduce depends on their rank, which induces a correlation between the particles.

\paragraph*{Outline of the paper} In the next section, we prove Theorem~\ref{thm:biinfinite} as well as a coupling-convergence result. In Section~\ref{sec:speed}, we prove Theorem \ref{thm:main} by linking the speed of the classical infinite-bin model to the one of the two-sided process. Finally, we prove Theorem~\ref{thm:analyticity} in Section~\ref{sec:analyticity} by showing that the length of the smallest good word in the past of the two-sided process has an exponential tail.

\section{Coupling words for the two-sided process}
\label{sec:coupling}

In this section we fix $\mu$ to be a probability distribution on $\N$. The proof of Theorem \ref{thm:biinfinite} is based on the existence of so-called coupling words, introduced by Chernysh and Ramassamy~\cite{CR} for the IBM($\mu$). More precisely, for every $K \in \Z_+$ we introduce the projection
\[
  \Pi_K: \begin{array}{rcl} S & \longrightarrow & \N^K\\ X & \longmapsto &\left( X\left(F(X)-K+1\right), \ldots, X\left(F(X) \right) \right) \end{array}
\]
which associates to a configuration $X$ its $K$\emph{-scenery seen from the front}, i.e. the number of balls in each of the rightmost $K$ non-empty bins. By convention, if $K=0$, the target of $\Pi_K$ is the singleton composed of the empty sequence. The \emph{coupling number} $\frakC(\gamma)$ of a word $\gamma\in\mathcal{W}$ is defined to be the largest integer $K\geq0$ such that after applying the moves in $\gamma$, the $K$-scenery seen from the front is independent of the starting configuration. More precisely,
\begin{equation}
\label{eq:defcouplingnumber}
  \frakC(\gamma)= \max\left\{ K \geq 0: \Pi_K(\Phi_\gamma(X)) = \Pi_K(\Phi_\gamma(Y)) \text{ for all } X, Y \in S\right\}.
\end{equation}
Since the set on the right-hand side of~\eqref{eq:defcouplingnumber} is an interval containing $0$, for all $0 \leq k \leq \frakC(\gamma)$, the image of the function $ \Pi_k \circ \Phi_\gamma$ is a singleton, i.e. the $k$-scenery seen from the front after applying the moves in $\gamma$ does not depend on the starting configuration.

For example, we have $\frakC(2,3,2,2)=1$, as one can check by distinguishing according to the two possible relative positions of the rightmost two balls in an arbitrary initial configuration. A word $\gamma$ is called $K$\emph{-coupling} if $\frakC(\gamma)\geq K$. If $\gamma$ is a $K$-coupling word, then any word which has $\gamma$ as a suffix is also a $K$-coupling word. Note however that a word having a $K$-coupling word $\gamma$ as a prefix may not be $K$-coupling. For example, we have $\frakC(2,3,2,2)=1$ and $\frakC(2,3,2,2,5)=0$. Nevertheless, we can control the variation of $\frakC$ when adding a suffix. If $\gamma\in\mathcal{W}$ and $a\in\N$, we recall that $\gamma\cdot a$ is the word of length $|\gamma|+1$ obtained by adding the letter $a$ to the end of $\gamma$.

\begin{lemma}
\label{lem:Cvariation}
Let $\gamma\in\mathcal{W}$ and $a\in\N$. Then $\frakC(\gamma\cdot a)\geq \frakC(\gamma)-1$. Furthermore if $a\leq \frakC(\gamma)$ then $\frakC(\gamma\cdot a) \geq \frakC(\gamma)$.
\end{lemma}

\begin{proof}
Denote by $M$ the number of balls in the single finite configuration in the image of $\Pi_{\frakC(\gamma)} \circ \Phi_\gamma$. The constant $M$ depends on the word $\gamma$, but not on the starting configuration. We distinguish two cases, whether $a \leq M$ or $a > M$.

We first assume $a \leq M$. In that case, after executing the moves corresponding to the letters of $\gamma$, the execution of $a$ selects a ball in the $K$th rightmost bin with $K\leq \frakC(\gamma)$, and places a ball in the bin immediately to the right of that bin. In particular, $\Pi_{\frakC(\gamma)} \circ \Phi_{\gamma\cdot a}(S)$ is still a singleton.

We now assume that $a>M$. Then, the execution of $a$ selects a ball in the $K$th rightmost bin with $K> \frakC(\gamma)$ ($K$ may depend on the initial configuration before the execution of $\gamma$) and places a ball in the bin immediately to the right of that bin. Note that while it might modify the content of the $\frakC(\gamma)$th rightmost bin, it does not change the content of any of the rightmost $\frakC(\gamma)-1$ bins. Thus $\Pi_{\frakC(\gamma)-1} \circ \Phi_{\gamma\cdot a}(S)$ is a singleton.

This proves that in any case, $\frakC(\gamma\cdot a) \geq \frakC(\gamma)-1$. Moreover, as there is necessarily at least one ball in each of the $\frakC(\gamma)$ rightmost bins, we know that $M \geq \frakC(\gamma)$. Therefore, if $a \leq \frakC(\gamma) \leq M$, then  $\frakC(\gamma\cdot a) \geq \frakC(\gamma)$.
\end{proof}

The following result will be the key for constructing a two-sided stationary version of the IBM and computing its speed.
\begin{proposition}
\label{prop:coupling}
Let $(\xi_n, n \in \Z)$ be a family of independent random variables with law $\mu$. For $K \in \N$, we set
\[
  \tau_K = \inf\{ n \geq 0: \xi_{-n}^0 \text{ is a } K\text{-coupling word} \}.
\]
Then $\tau_K$ is finite a.s.
\end{proposition}

\begin{proof}
Let $a$ be the smallest integer in the support of $\mu$. Setting $m = \frac{a(a-1)}{2}+1$, we denote by $a^m$ the word of length $m$ containing only letters $a$. We first show that applying $\Phi_{a^m}$ to any initial configuration has the effect of making the front advance by at least $1$. This can be observed using the partial order $\preccurlyeq$ on $S$ introduced in~\cite[Section 2]{MR}, which is such that for any $X \preccurlyeq X'$ in $S$, $F(X) \leq F(X')$ and for any word $\gamma \in \mathcal{W}$, $\Phi_{\gamma}(X) \preccurlyeq \Phi_{\gamma}(X')$. For any $n\in\Z$, the smallest configuration in $S$ with the front at position $n$ is $\underline{X}^{(n)}(k) = \ind{k \leq n}$ and one easily checks that applying $\Phi_{a^m}$ to $\underline{X}^{(n)}$ has the effect of making the front advance by $1$. Therefore, applying $\Phi_{a^m}$ to any configuration makes the front advance by at least $1$. We will need this observation towards the end of the proof.

Note that for any $K < K'$, a $K'$-coupling word is also a $K$-coupling word, hence $\tau_K \leq \tau_{K'}$. Therefore we can without loss of generality choose the integer $K$ as large as we wish in this proof. We introduce the following sequence of waiting times (backward in time) defined by $T_0=0$ and
\[
  T_{k+1} = \sup \left\{ n < T_k: \xi_n \geq K \quad \text{or} \quad (n+ m-1 < T_k \text{ and } \xi_n^{n+m-1}=a^m) \right\}.
\]
We now choose $K$ large enough such that
\begin{equation}
  \label{eqn:assumptionK}
  \P(\xi_{T_{1}} \geq K) < 1/3
\end{equation}
i.e. such that it is at least twice less likely to observe a letter larger than $K$ than to observe $m$ occurrences of $a$ in a row when observing the sequence $(\xi_n)$. Indeed, one can check that
\[
  \P(\xi_{T_1} \geq K) = \frac{\mu([K,\infty)) ( 1 - \mu(a)^{m+1}  )}{\mu(([K,\infty)) + (\mu([K,\infty)) + \mu(a)) \mu(a)^{m+1}},
\]
which can be made as small as wished as $K \to \infty$, hence for $K$ large enough, assumption \eqref{eqn:assumptionK} can be verified.

We denote by $S_0=0$ and $S_{k+1} = S_k + 2 \ind{\xi_{T_{k+1}}\geq K} - 1$ an associated random walk. For all $k \geq 0$, $T_{k+1}$ is the first time before $T_k$ where we see either a letter larger than $K$ or the pattern $a^m$ and $S_k$ counts the difference between the number of times the former versus the latter occurs. By assumption \eqref{eqn:assumptionK}, we have $\E(S_1) < -1/3$, thus $(S_k)$ drifts towards $-\infty$. As a result, we know there exists an infinite sequence of times $(R_k, k \geq 0)$ defined by $R_0=0$ and $R_{k+1} = \inf\{ n > R_k: S_n < S_{R_k} \}$, the time at which $S_n$ reaches its record minimum for the $(k+1)$st time. 

Let $b$ be the second largest integer in the support of $\mu$ (here we use the fact that $\mu$ is non-degenerate). Theorem 1.1 in \cite{CR} says exactly that for any $N\geq1$ there exists a word $\gamma'_N$ using only letters $a$ and $b$ such that after applying the moves in $\gamma'_N$, the position of the rightmost $N$ balls relatively to the front is independent of the starting configuration. A straightforward reformulation of that result is that for any $N\geq1$ there exists a word $\gamma_N$ using only letters $a$ and $b$ such that after applying the moves in $\gamma_N$, the content of the rightmost $N$ nonempty bins is independent of the starting configuration. Hence, there exists a $(K+1)$-coupling word $\gamma$ that is written only using the letters $a$ and $b$. Note that if $a=1$ (i.e. $\mu(1) \neq 0$), the word $\gamma$ can be chosen to be the word obtained by repeating $K+1$ times the letter $1$. We define the new waiting time
\[
  L= \inf\left\{ k \in \N: \xi_{T_{R_k} - |\gamma|}^{T_{R_k}-1} = \gamma\right\},
\]
i.e. the first time that the word $\gamma$ appears immediately before a time at which the random walk $S$ hits a new minimum. As the appearance of the word $\gamma$ immediately before time $T_{R_k}$ has positive probability of occurring and is independent of everything that happens after time $T_{R_k}$, we observe that $L<\infty$ a.s. 

Set $N= -T_{R_L}+|\gamma|$. To conclude the proof, it is enough to show that $\xi_{-N}^{0}$ is a $K$-coupling word, which will prove that $\tau_K \leq N < \infty$ a.s. To do so, we prove that for any $ T_{R_L}-1 \leq n < 0$, we have $\frakC(\xi_{-N}^n)\geq K+1$. It is true for $n= T_{R_L}-1$, since $\xi_{-N}^{T_{R_L}-1}=\gamma$ which is a $(K+1)$-coupling word. For any $k\geq1$, define
\[
T'_k=
\begin{cases}
T_k &\text{ if } \xi_{T_k}\geq K \\
T_k+m-1 &\text{ otherwise.}
\end{cases}
\]
When reading the word $\xi_{-N}^0$ from left to right, the time $T'_k$ is the $(R_L +1 - k)$th time that we read either a letter larger or equal to $K$ or the rightmost letter of a pattern $a^m$. We also set $T'_0=0$.
One shows by induction on $0 \leq k < R_L$ that for any $T'_{k+1} \leq n < T'_k$, we have
\[
\frakC\left(\xi_{-N}^n\right) \geq \frakC\left(\xi_{-N}^{T_{R_L}-1}\right) + S_{k+1} - S_{R_L}.
\]
This is a consequence of Lemma~\ref{lem:Cvariation} and the facts that every letter at least $K$ decreases the coupling number by at most one, every pattern $a^m$ increases the coupling number by at least one (by the observation made at the beginning of the proof) and all the other patterns do not decrease the coupling number, since this coupling number stays above $K$. We conclude from the fact that $S_{k+1}\geq S_{R_L}$ if $0 \leq k < R_L$, as $R_L$ is a time when the random walk $S$ hits its minimum.
\end{proof}

Using the a.s. existence of finite times $\tau_K$ for $K \in \N$, we deduce Theorem~\ref{thm:biinfinite}.

\begin{proof}[Proof of Theorem~\ref{thm:biinfinite}]
We construct the configuration $Y_0$ as follows. For each $K \in \N$, we set the rightmost $K$ non-empty bins of $Y_0$ to be the single configuration in $\Pi_K \circ \Phi_{\xi_{-\tau_K}^0}(S)$, which is a.s. well-defined as $\tau_K < \infty$ a.s.

This construction is consistent for different values of $K$ and it produces a unique configuration $Y_0$ by sending $K$ to infinity and requiring that $F(Y_0)=0$. The variable $\Pi_K(Y_0)$ is measurable with respect to $\xi^0_{-\tau_K}$, so $Y_0 \in \mathcal{F}_0$ a.s. As a result, for any $n > 0$, $Y_n = \Phi_{\xi_1^n}(Y_0)$ is a.s. $\mathcal{F}_n$-measurable.

If $n < 0$, we can do a similar analysis as the one made for $n=0$. For any $K \in \N$, $\Pi_K(Y_n) \in \mathcal{F}_n$ a.s. Choosing $K>-n$ and using Lemma~\ref{lem:Cvariation}, one can deduce $F(Y_n)$ from $\Pi_K(Y_n)$ and from $\xi_n^0$, since we know that $F(Y_0)=0$. We conclude that the configuration $Y_n$ is a.s. entirely determined (up to a shift) by the sequence $(\xi_k, k\leq n)$.
\end{proof}

For any $k \in \Z$, the law of any $K$-scenery seen from the front of $Y_k$ depends only on $(\xi_n, n \leq k)$, which has the same law as $(\xi_n, n \leq 0)$. Hence the law of $\Psi_{F(Y_k)}(Y_k)$, which is the configuration $Y_k$ shifted to place its front at position $0$, is indeed the same as the law of $Y_0=\Psi_{F(Y_0)}(Y_0)$, as claimed in the paragraph right after the statement of Theorem~\ref{thm:biinfinite}.

Now that we have constructed the two-sided process $Y$, we observe using similar methods as in Proposition~\ref{prop:coupling}, that any infinite-bin model $X$ ends up behaving like this two-sided process.
\begin{proposition}[Coupling-convergence]
\label{prop:couplingconvergence}
Let $(\xi_n, n \in \Z)$ be i.i.d. random variables with law $\mu$ and fix $X_0 \in S$. We denote by $(Y_n)_{n \in \Z}$ the two-sided process defined in Theorem \ref{thm:biinfinite}, and by $(X_n)_{n\geq0}$ the infinite-bin model constructed with $X_0$ and the random variables $(\xi_n, n \geq 1)$. For any $K \in \N$, for all $n \geq 0$ large enough, we have
\[
  \Pi_K(X_n) = \Pi_K(Y_n) \quad \text{a.s.}
\]
\end{proposition}

Note that this proposition in particular implies the convergence of the $K$-scenery seen from the front for any infinite-bin model $X$.

\begin{proof}
Let $a<b$ be the smallest two integers in the support of $\mu$ and let $K>0$ be a number large enough such that \eqref{eqn:assumptionK} holds. We then define the sequence of waiting times $T_0=0$ and
\[
  T_{k+1} = \inf\left\{ n > T_k: \xi_n \geq K \quad \text{or} \quad (n-m+1>T_k \text{ and } \xi_{n-m+1}^{n}=a^m) \right\}.
\]
The random walk $S_k:=\sum_{j=1}^k (2\ind{\xi_{T_k} \geq K} - 1)$ drifts to $-\infty$. In particular, for infinitely many integers $n\geq1$, we have $\sup_{k \geq n} S_k \leq S_n$.

Let $\gamma$ be a $(K+1)$-coupling word consisting only of the letters $a$ and $b$. Almost surely, there exists a time $N$ large enough such that the first $|\gamma|$ letters after $N$ spell the word $\gamma$ and the random walk $S$ observed after time $N + |\gamma|$ is always below its value at time $N + |\gamma|$.
Hence, by an argument similar to the one used to prove Proposition~\ref{prop:coupling}, for all $n \geq N+|\gamma|$, we have $\Pi_K(X_n) = \Pi_K(Y_n)$, which concludes the proof.
\end{proof}

\section{Speed of the infinite-bin model}
\label{sec:speed}

In this section, we use the stationary infinite-bin model $Y$ we constructed in the previous section to obtain formula~\eqref{eq:speedformula} for the speed of the infinite-bin model.
\begin{lemma}
\label{lem:useStationarity}
Let $\mu$ be a non-degenerate probability law on $\N$ and $(Y_n)_{n\in\Z}$ be a two-sided stationary infinite-bin model. We have $v_\mu = \P(F(Y_1)=1)$.
\end{lemma}

\begin{proof}
Since $(Y_n)_{n\geq0}$ is an infinite-bin model with move distribution $\mu$, by \cite[Theorem 1.1]{MR} and dominated convergence, we have
\[
  v_\mu = \lim_{n \to \infty} \frac{1}{n} \E(F(Y_n)) = \lim_{n \to \infty} \frac{1}{n} \sum_{j=1}^n \E(F(Y_j) - F(Y_{j-1})).
\]
By the stationary property of $Y$, we also observe that
\[
\E(F(Y_j)-F(Y_{j-1})) = \E(F(Y_1)-F(Y_0)) = \P(F(Y_1)=1),
\]
thus $v_\mu = \P(F(Y_1)=1)$.
\end{proof}

We use this expression for $v_\mu$ in terms of the two-sided process $Y$ to prove Theorem~\ref{thm:main}.

\begin{proof}[Proof of Theorem~\ref{thm:main}]
Let $(\xi_n, n \in \Z)$ be i.i.d. random variables with law $\mu$. We introduce the random time
\[
  T= \inf\{ n \geq 0: \xi_{-n}^1 \in \mathcal{G} \cup \mathcal{B} \}.
\]
We first note that if $\xi_{-n}^0$ is a $1$-coupling word, then we know the value of $Y_0(0)$. In that case, $\xi_{-n}^1$ is either a good word or a bad word, depending on whether $\xi_1 \leq Y_0(0)$ or not. We conclude that $T <\tau_1 < \infty$ a.s. by Proposition~\ref{prop:coupling}. Lemma~\ref{lem:useStationarity} yields
\[
  v_\mu = \P(F(Y_1)=1) = \P( \xi_{-T}^1 \in \mathcal{G} ) = 1 - \P(\xi_{-T}^1 \in \mathcal{B}).
\]
Note that if $\xi_{-T}^1$ is good, then it is necessarily a minimal good word (if it had a good strict suffix, $T$ would have been smaller). Similarly, $\xi_{-T}^1\in\mathcal{B}$ implies $\xi_{-T}^1\in\mathcal{B}_m$. Moreover the support of $\xi_{-T}^1$ is the entire set $\mathcal{G}_m \cup \mathcal{B}_m$, as a good (resp. bad) word cannot have a bad (resp. good) suffix. Thus
\begin{align*}
  \P(\xi_{-T}^1 \in \mathcal{G}) = \P(\xi_{-T}^1 \in \mathcal{G}_m) &= \sum_{\alpha \in \mathcal{G}_m} \P(\xi_{-T}^1 = \alpha)\\
  &= \sum_{\alpha \in \mathcal{G}_m} \P(\xi_{-|\alpha|+2}^1 = \alpha) = \sum_{\alpha \in \mathcal{G}_m} w_\mu(\alpha),
\end{align*}
hence $v_\mu= \sum_{\alpha \in \mathcal{G}_m} w_\mu(\alpha)$. The equality $v_\mu=1- \sum_{\alpha \in \mathcal{B}_m} w_\mu(\alpha)$ follows from similar computations.
\end{proof}

\begin{remark}
\label{rem:criterion}
In order to make the formulas in Theorem~\ref{thm:main} effective, one needs a criterion to find the minimal good and bad words. Given a word $\alpha\in\mathcal{W}$, it suffices to test it against a finite set $\Sigma$ of configurations to determine whether it is good or bad: if $\alpha$ is $X$-good (resp. $X$-bad) for every $X\in\Sigma$, then it is good (resp. bad). Writing
\[
h= \max_{1 \leq i \leq |\alpha|} 1 + \alpha_i - i,
\]
the set $\Sigma$ can be taken to be any set of $2^{h-1}$ configurations with the front at position $0$ such that for any $X\neq X'$ in $\Sigma$, the positions of the rightmost $h$ balls in $X$ and $X'$ are not all the same.
\end{remark}

\section{Analyticity of \texorpdfstring{$C(p)$}{C(p)}}
\label{sec:analyticity}

Using the formula we obtained for the speed $v_\mu$, we are now able to prove the analyticity of the growth rate $C$ of the length of the longest path in Barak-Erd\H{o}s graphs.

\begin{proof}[Proof of Theorem \ref{thm:analyticity}]
For any $p,q \geq 0$, we write
\begin{equation}
\label{eq:bivariateseries}
  D(p,q)= \sum_{\alpha \in \mathcal{G}_m} p^{|\alpha|} q^{\sum_{j=1}^{|\alpha|} (\alpha_j - 1)}.
\end{equation}
As stated in Corollary~\ref{cor:Cpformula}, it follows from the coupling of Foss and Konstantopoulos \cite{FK} between infinite-bin models and Barak-Erd\H{o}s graphs that for any $0 < p\leq 1$,
\[
C(p)=v_{\mu_p}=D(p,1-p), 
\]
where $\mu_p$ denotes the geometric distribution of parameter $p$.

To prove that $C$ is analytic around some $p_0\in(0,1]$, it is enough to show that the series~\eqref{eq:bivariateseries} converges for some pair $(p',q')$ with $p'>p_0$ and $q'>1-p_0$. Indeed, one would then deduce that all the series of derivatives of $D(p,1-p)$ converge normally around $p_0$.

Recall that $T = \inf\{ n \geq 0: \xi_{-n}^1 \in \mathcal{G} \cup \mathcal{B} \}$. For any probability distribution $\mu$ on $\N$, we denote by $\E_{\mu}$ the expectation associated with the IBM($\mu$). By simple computations similar to the proof of Theorem~\ref{thm:main}, for any $r > 0$ we have
\[
  \E_{\mu_p}(r^{T+2}) \geq \E_{\mu_p}(r^{T+2} \ind{\xi_{-T}^1 \in \mathcal{G}}) = D(rp,1-p),
\]
As a result, to conclude the proof, it is enough to show that $T$ admits some exponential moments uniformly in $p$. More precisely, we will prove that for every $s \in (0,1]$, there exists $r_s>1$ such that
\begin{equation}
  \label{eqn:exponentialMoments}
  \forall p \in [s,1], \E_{\mu_p}(r_s^T) < \infty.
\end{equation}
Then, for any $0 < p_0 \leq 1$, choosing $p$ such that
\[
\max\left(\frac{p_0}{2},\frac{p_0}{r_{p_0/2}}\right) < p < p_0
\]
and setting $p'=pr_{p_0/2}$ and $q'=1-p$, one obtains the convergence of the series $D(p',q')$, which will prove Theorem~\ref{thm:analyticity}.

Recall from the proof of Theorem~\ref{thm:main} that $T$ is smaller than $\tau_1$, the smallest time such that $\xi_{-\tau_1}^0$ is a $1$-coupling word. To bound $\E_{\mu_p}(r^{\tau_1})$ we use a construction similar to the one in the proof of Proposition~\ref{prop:coupling}. Fix $s\in(0,1]$. We choose an integer $K \geq 1$ large enough such that $2 (1-s)^{1-K} \leq s$. Then for every $p \in [s,1]$, we have
\[
  2 \mu_p([K,\infty)) \leq 2 (1-s)^{1-K} \leq s \leq \mu_p(1).
\]
We now introduce the sequence defined by $T_0=0$ and for any $k\geq0$,
\[
  T_{k+1} = \sup\{ n < T_k: \xi_n = 1 \quad \text{or} \quad \xi_n \geq K\}.
\]
We also set $S_k = -k + 2\sum_{j = 1}^k \ind{\xi_{T_j} \geq K} $. For every $p \in [s,1]$, by the choice of $K$, $S$ is a nearest-neighbor random walk such that $\E(S_1) \leq - 1/3$. We denote by $(R_k, k \geq 0)$ the sequence of strictly descending ladder times of $S$ (i.e. $R_k$ is the $k$th time when $S$ reaches its record minimum) and by $\gamma$ the word consisting in $K+1$ times the letter $1$. Then, setting
\[
  L= \inf\left\{ k \in \N: \xi_{T_{R_{k(K+1)}}-K-1}^{T_{R_{k(K+1)}}-1} = \gamma \right\},
\]
we have $\tau_1 \leq -T_{R_{L(K+1)}} + K+1$.

As $(S_k)$ is a random walk with negative drift smaller than $-1/3$, for any $k \in \N$, the random variable $R_k$ is stochastically dominated by the sum of $k$ i.i.d. random variables $U_j$ with the law of the first hitting time of $-1$ by a nearest-neighbor random walk with drift $-1/3$. Moreover, we observe that
\[
  \E(r^{U_j}) = \frac{2}{3}r + \frac{1}{3}r \E(r^{U_j})^2,
\]
by decomposition with respect to the first step of the random walk. Therefore, for any $k \in \N$, $p \in [s,1]$ and $r < \frac{3}{2\sqrt{2}}$, we have
\[
  \E_{\mu_p}\left( r^{R_k} \right) \leq \left( \frac{3  - \sqrt{9-8r^2}}{2r} \right)^k.
\]

Similarly, we observe that for any $p \in [s,1]$, $-T_1$ is stochastically dominated by a geometric random variable with parameter $s$, as this is the minimal probability for obtaining a $1$. Then $-T_k$ is stochastically dominated by the sum of $k$ i.i.d. copies of a geometric random variable with parameter $s$. Thus, by conditioning with respect to $R_k$, as long as $1 < r < \frac{1}{1 - \frac{(3-2\sqrt{2})}{3}s}$, for all $p \in [s,1]$ and $k \in \N$ we have
\[
 \E_{\mu_p}(r^{-T_{R_k}}) \leq \E_{\mu_p}\left(\left(\E_{\mu_s}\left( r^{-T_1} \right)\right)^{R_k} \right) \leq \left( \frac{3  - \sqrt{9-8\left(\tfrac{sr}{1-(1-s)r}\right)^2}}{2\tfrac{sr}{1-(1-s)r}} \right)^k.
\]
Finally, as $L$ can be stochastically dominated by a geometric random variable with parameter $s^{K+1}$, which is independent of $(T_{R_k}, k \geq 1)$, it admits some finite exponential moments. We conclude that~$\tau_1$ also admits some exponential moments, uniformly in $p \in [s,1]$.
\end{proof}

We point out that, by computations similar to those above, one could show that for any probability distribution $\mu$ and for any $K \in \N$, there exists $r>1$ such that $\E_{\mu}(r^{\tau_K})<\infty$, hence $\P(\tau_K > n)$ decays exponentially fast with $n$.

\paragraph*{Acknowledgements}

We would like to thank an anonymous referee of our previous paper~\cite{MR} for a comment inspiring us to study the set of good words. S.R. acknowledges the support of the Fondation Simone et Cino Del Duca.

\label{Bibliography}
\bibliographystyle{plain}
\bibliography{bibliographie}

\Addresses

\end{document}